\documentclass{amsart}

\usepackage{
 amsmath, 
 amsxtra, 
 amsthm, 
 amssymb, 
 etex, 
 mathrsfs, 
 mathtools, 
 tikz-cd, 
 xr,
 fullpage,
 comment}
\usepackage[all]{xy}
\usepackage{hyperref}

\newtheorem{theorem}{Theorem}[section]
\newtheorem{lemma}[theorem]{Lemma}
\newtheorem{conjecture}[theorem]{Conjecture}
\newtheorem{proposition}[theorem]{Proposition}
\newtheorem{corollary}[theorem]{Corollary}
\newtheorem{defn}[theorem]{Definition}

\newtheorem{lthm}{Theorem} 

\theoremstyle{remark}
\newtheorem{remark}[theorem]{Remark}

\newcommand{\QQ}{\mathbb{Q}}
\newcommand{\Qp}{\mathbb{Q}_p}

\newcommand{\ZZ}{\mathbb{Z}}
\newcommand{\Zp}{\mathbb{Z}_p}

\newcommand{\DD}{\mathbb{D}}

\newcommand{\cY}{\mathcal{Y}}

\DeclareMathOperator{\Gal}{Gal}
\DeclareMathOperator{\Fil}{Fil}

\DeclareMathOperator{\Ext}{Ext} 

\DeclareMathOperator{\Sel}{Sel}
\DeclareMathOperator{\coker}{coker}
\DeclareMathOperator{\length}{length}
\DeclareMathOperator{\rank}{rank}

\newcommand{\ord}{\mathrm{ord}}

\newcommand{\bc}{\mathbf{c}}
\newcommand{\vp}{\varphi}
\newcommand{\Iw}{\mathrm{Iw}}
\newcommand{\HIw}{H^1_{\Iw}}
\newcommand{\col}{\mathrm{Col}}
\newcommand{\image}{\mathrm{Im}}
\newcommand{\loc}{\mathrm{loc}}

\newcommand{\LL}{\Lambda}
\newcommand{\lra}{\longrightarrow}
\newcommand{\res}{\textup{res}}

\newcommand{\pr}{\textup{pr}}
\newcommand{\cor}{\mathrm{cor}}
\newcommand{\Mlogv}{M_{\log,v}}

\newcommand{\Sss}{{\Sigma_\mathrm{ss}}}
\newcommand{\Sssp}{{\Sigma'_\mathrm{ss}}}
\newcommand{\Sord}{\Sigma_\mathrm{ord}}
\newcommand{\fs}{{\vec{s}}}
\newcommand{\Dcrisw}{\DD_{\mathrm{cris},w}(T)}
\newcommand{\Dcrisv}{\DD_{\mathrm{cris},v}(T)}
\newcommand{\Ep}{E[p^\infty]}
\newcommand{\Selp}{\Sel_{p^\infty}}

\newcommand{\colsv}{\mathrm{Col}_{\sharp,v}}
\newcommand{\colfv}{\mathrm{Col}_{\flat,v}}
\newcommand{\colsw}{\mathrm{Col}_{\sharp,w}}
\newcommand{\colfw}{\mathrm{Col}_{\flat,w}}
\newcommand{\colbw}{\mathrm{Col}_{\bullet,w}}
\newcommand{\hvnu}{h_{v,n}^{u}}

\newcommand{\colvnu}{\mathrm{Col}_{v,n}^{u}}
\newcommand{\colwnu}{\mathrm{Col}_{w,n}^{u}}

\newcommand{\cH}{\mathcal{H}}
\newcommand{\cL}{\mathcal{L}}
\newcommand{\cO}{\mathcal{O}}

\newcommand{\cX}{\mathcal{X}}
\newcommand{\wT}{\widehat{T}}
\newcommand{\tT}{\widetilde{T}}

  \DeclareFontFamily{U}{wncy}{}
  \DeclareFontShape{U}{wncy}{m}{n}{<->wncyr10}{}
  \DeclareSymbolFont{mcy}{U}{wncy}{m}{n}
  \DeclareMathSymbol{\sha}{\mathord}{mcy}{"58} 

\newcommand{\plim}{\displaystyle \mathop{\varprojlim}\limits}

\definecolor{Green}{rgb}{0.0, 0.5, 0.0}
\definecolor{blue}{rgb}{0.0, 0.0, 1.0}

\newcommand{\draftcolor}{Green}

\newcommand{\bgreen}{\begin{color}{\draftcolor}}
\newcommand{\egreen}{\end{color}}

\begin{document}
\title[Growth of Mordell-Weil ranks and sha with  mixed-reduction type]{Mordell-Weil ranks and Tate-Shafarevich groups of elliptic curves with mixed-reduction type over cyclotomic extensions}

\begin{abstract}
Let $E$ be an elliptic curve defined over a number field $K$ where $p$ splits completely. Suppose that $E$ has good reduction at all primes above $p$. Generalizing previous works of  Kobayashi and Sprung, we define multiply signed Selmer groups over the cyclotomic $\Zp$-extension of a finite extension $F$ of $K$ where  $p$ is unramified. Under the hypothesis that the Pontryagin duals of these Selmer groups are torsion over the corresponding Iwasawa algebra, we show that the Mordell-Weil ranks of $E$ over a subextension of the cyclotomic $\Zp$-extension are bounded. Furthermore, we derive an aysmptotic formula of the growth of the $p$-parts of the Tate-Shafarevich groups of $E$ over these extensions.
\end{abstract}

\author[A. Lei]{Antonio Lei}
\address[Lei]{D\'epartement de Math\'ematiques et de Statistique\\
Universit\'e Laval, Pavillion Alexandre-Vachon\\
1045 Avenue de la M\'edecine\\
Qu\'ebec, QC\\
Canada G1V 0A6}
\email{antonio.lei@mat.ulaval.ca}

\author[M.F.~Lim]{Meng Fai Lim}
\address[Lim]{School of Mathematics and Statistics \& Hubei Key Laboratory of Mathematical Sciences\\ Central China Normal University\\ Wuhan\\ 430079\\ P.R.China.}
\email{limmf@mail.ccnu.edu.cn}

\numberwithin{equation}{section}

\thanks{The authors' research is partially supported by:  the NSERC Discovery Grants Program RGPIN-2020-04259 and RGPAS-2020-00096 (Lei) and the National Natural Science Foundation of China under Grant No. 11550110172 and Grant No. 11771164 (Lim).}

\subjclass[2010]{11R23 (primary); 11F11, 11R20 (secondary) }
\keywords{Iwasawa theory, elliptic curves, Mordell-Weil groups, Tate-Shafarevich groups, mixed-reduction type}

\maketitle

\section{Introduction}
Let $p$ be a fixed odd prime number and  $K\subset F$  be number fields. Let $\Sigma_p'$ and $\Sigma_p$ be the sets of primes of $K$ and $F$ above $p$ respectively. Throughout this article, we fix an elliptic curve $E/K$ which has good reduction at all primes of $\Sigma_p'$. We decompose $\Sigma_p'$ and $\Sigma_p$ into the ordinary and supersingular primes for $E$, namely $\Sigma_p'=\Sssp\sqcup\Sord'$ and $\Sigma_p=\Sss\sqcup\Sord$. Without further notice, we assume throughout the article that the following hypotheses hold:
\begin{itemize}
    \item[(S1)] The prime $p$ splits completely in $K/\QQ$ and is unramified in $F/\QQ$;
    \item[(S2)] The set $\Sssp$ is non-empty.
    \end{itemize}

Denote by $K_\infty$ the cyclotomic $\Zp$-extension of $K$. For $n\ge0$, let $K_n$ denote the unique sub-extension of $K_\infty/K$ with $[K_n:K]=p^n$. Similarly, we write $F_\infty$ for the cyclotomic $\Zp$-extension of $F$ and $F_n$ for the unique sub-extension of $F_\infty/F$ with $[F_n:F]=p^n$. In view of assumption (S1), we have $K_{\infty}\cap F_n = K_n$ for every $n$.
For each $w\in \Sigma_p$,  the unique place of $F_\infty$ lying above $w$ will again be denoted by $w$. For each $w\in\Sss$, we define two local conditions $E^\sharp(F_{\infty,w})$ and $E^\flat(F_{\infty,w})$, generalizing works of Kobayashi \cite{kobayashi03} and Sprung \cite{sprung09}. This allows us to define $2^{|\Sss|}$ multiply signed Selmer groups $\Sel^{\fs}(E/F_\infty)$, one for each choice of $\fs=(s_w)_{w\in\Sss}$, where $s_w\in\{\sharp,\flat\}$. Our construction is carried out in \S\ref{S:signed}. Let $\Lambda$ denote the Iwasawa algebra $\Zp[[\Gal(F_\infty/F)]]$. In the main body of the article, we will introduce an additional hypothesis affirming that the Pontryagin duals of the multiply signed Selmer groups are $\Lambda$-torsion (labelled (S3)). We shall write $\mu_{\fs}$ and $\lambda_{\fs}$ for the $\mu$- and $\lambda$-invariants of $\Sel^{\fs}(E/F_\infty)^\vee$.

Our first result is a uniform  bound on the Mordell-Weil ranks of $E$ over $F_n$ as $n$ grows.

\begin{lthm} \label{thm:ranks} Under hypotheses (S1)-(S3),  $\rank_{\ZZ} E(F_n)$ is bounded independently of $n$.
\end{lthm}

When $F/\QQ$ is an abelian extension, one may obtain this result using Kato's Euler system from \cite{kato04} (together with the non-vanishing of the $L$-values of $E$ proved by Rohrlich \cite{rohrlich88}). Our method does not assume the existence of an Euler system and relies on the cotorsionness of the multiply signed Selmer groups instead. See also \cite[Corollary~10.2]{kobayashi03}, \cite[Theorem~3.4]{LP19}, \cite[Proposition~5.4]{LLZ17} and \cite[Theorem~1.1]{LS19} for similar results.

Given a finite $p$-group $M$, we write $e(M)$ for the integer given by $|M|=p^{e(M)}$.  The second result of our article is about the growth of $p$-parts of  Tate-Shafarevich groups of $E$ over $F_n$ (that is, $e\left(\sha_p(E/F_n)\right)$), as $n$ grows.

\begin{lthm} \label{thm:sha}Suppose that the hypotheses (S1)-(S3) hold. Furthermore, suppose that $\sha_p(E/F_n)$ is finite for all $n$. Then, there exist a choice of $\vec{\sigma}$ and $\vec{\tau}$ in $\{\sharp,\flat\}^\Sss$ such that
\[e\left(\sha_p(E/F_n)\right)-e\left(\sha_p(E/F_{n-1})\right)=
\begin{cases}
    S(\vec{\sigma},n)+\phi(p^n)\mu_{\vec{\sigma}}+\lambda_{\vec{\sigma}}-r_\infty&\text{if $n$ is odd,}\\
    T(\vec{\tau},n)+\phi(p^n)\mu_{\vec{\tau}}+\lambda_{\vec{\tau}}-r_\infty&\text{if $n$ is even.}
    \end{cases}
\]
for all $n\gg0$, where $r_\infty=\lim_{n\rightarrow\infty} \rank_{\ZZ}E(F_n)$, $S(\vec{\sigma},n)$ and $T(\vec{\tau},n)$ are certain linear combinations of $p^i$, $i\le n$, which we define explicitly in Proposition~\ref{prop:Y''} and $\phi$ is the Euler totient function.
\end{lthm}
This generalizes results of Kurihara, Kobayashi and Pollack for elliptic curves defined over $\QQ$ with $a_p=0$ (see \cite[Theorem~0.1]{kurihara02}, \cite[Theorem~1.4]{kobayashi03} and \cite[Theorem~1.1]{pollack05}) as well as Sprung's result for general $a_p$ (see \cite[Theorem~1.1]{sprung13}).

\begin{remark}
In the case where $a_v=0$ for all $v\in\Sssp$, the formula of Theorem~\ref{thm:sha} simplifies to 
\[
\begin{split}
    e\left(\sha_p(E/F_n)\right)&-e\left(\sha_p(E/F_{n-1})\right)=\\
&\begin{cases}
   \sum_{v\in\Sss}[F_v:\Qp](p^{n-1}-p^{n-2}+p^{n-3}-\cdots -p)+\phi(p^n)\mu_{\vec{\flat}}+\lambda_{\vec{\flat}}-r_\infty&\text{if $n$ is odd,}\\
   \sum_{v\in\Sss}[F_v:\Qp](p^{n-1}-p^{n-2}+p^{n-3}-\cdots -1)+\phi(p^n)\mu_{\vec{\sharp}}+\lambda_{\vec{\sharp}}-r_\infty&\text{if $n$ is even,}
    \end{cases}
\end{split}
\]
where $\vec{\star}$ denotes the constant vector $(\star)_{w\in\Sss}$ for $\star\in\{\sharp,\flat\}$. {In particular, the  vectors $\vec\sigma$ and $\vec\tau$ in the statement of Theorem~\ref{thm:sha} are given by $\vec\flat$ and $\vec\sharp$ respectively.}

If furthermore $\Sigma_p=\Sss$, the term $\sum_{v\in\Sss}[F_v:\Qp]$ becomes $[F:\QQ]$. In this case, under certain hypotheses on the vanishing of the Mordell-Weil ranks and the behaviour of $\sha_p(E/F)$, Iovita and Pollack \cite[Theorem~5.1]{iovitapollack06}  showed that the quantities $[F:\QQ](p^{n-1}-p^{n-2}+\cdots)$  describe precisely the growth of the Tate-Shafarevich groups of $E$ over finite extensions inside a $\Zp$-extension of $F$ (which is not necessarily cyclotomic).  It would be interesting to study whether our techniques can be extended to the setting of \cite{iovitapollack06}, which may allow us to relax some of the hypotheses in loc. cit. We plan to study this in the near future.
\end{remark}

The structure of the paper is as follows. In \S\ref{S:signed}, we review the local theory of Coleman maps and logarithmic matrices for elliptic curves with supersingular reduction at $w\in\Sss$. This allows us to define the multiply signed Selmer groups. Along the way, we prove a result on the image of the direct sum of two Coleman maps (Proposition~\ref{prop:image}), which may be of independent interest. After giving the definition of multiply signed Selmer groups, we prove a number of results on the structure of global cohomology groups under the hypothesis that these Selmer groups are cotorsion. In \S\ref{S:kob}, we first review the definition of Kobayashi ranks on projective systems of $\Zp$-modules. We then prove a number of preliminary results on Kobayashi ranks of certain modules that will be used later on in the article. In \S\ref{S:local}, we make the link between Coleman maps and Kobayashi ranks and explain how this allows us to study the growth of certain local modules. We treat the ordinary and supersingular cases separately. Our treatment  in the supersingular case follows closely \cite{sprung13,LLZ17}. Finally, we put everything together to prove Theorems~\ref{thm:ranks} and \ref{thm:sha} in \S\ref{sec:proofs}.

\subsection*{Acknowledgement} We would like to thank Antonio Cauchi, Daniel Delbourgo, Jeffrey Hatley, Chan-Ho Kim and Guhan Venkat for interesting discussions during the preparation of this article. Some part of the research of this article was conducted when Lim was visiting the National University of Singapore and the National Center for Theoretical Sciences in Taiwan, and he would like to acknowledge the hospitality and conducive working conditions provided by these institutes. Finally, we thank the anonymous referees for their very helpful comments, which have helped improve the presentation of the article.

\subsection{Notation}

Throughout this article, $T$ denotes the $p$-adic Tate module of $E$.

The Galois groups of $F_\infty/F$ and $F_\infty/F_n$ are denoted by $\Gamma$ and $\Gamma_n$ respectively. We fix once and for all a topological generator $\gamma$ of $\Gamma$. Recall from our earlier discussion that we write the unique prime of $F_\infty$ above  a prime $w\in \Sigma_p$ by $w$ as well. In particular, it follows that $F_{\infty,w}$ is the cyclotomic $\Zp$-extension of $F_w$ with $F_{n,w}$ as its intermediate subfields. Therefore, we may and will identify the Galois groups of $F_{\infty,w}/F_w$ and $F_{\infty,w}/F_{n,w}$ with $\Gamma$ and $\Gamma_n$ respectively. 
Let $\Lambda$ denote the Iwasawa algebra $\Zp[[\Gamma]]$, which we shall identify with the power series ring $\Zp[[X]]$ by sending $\gamma-1$ to $X$.

 We write $\cH$ for the ring of distribution algebra on $\Gamma$, which can be realized as the set of power series in $\Qp[[X]]$ that converge on the open unit disc.
Let  $G$ be an element in $\Lambda$ or $\cH$. We shall identify it with a power series in $\Qp[[X]]$ (which will again be denoted by $G$). Given  a character $\theta$ on $\Gamma$, we evaluate $G$ at $\theta$ via $G(\theta)=G(\theta(\gamma)-1)$.

Given a ring $\cO$ that contains $\Zp$, we shall write $\Lambda_\cO$ for the tensor product $\Lambda\otimes_{\Zp}\cO$. We may evaluate an element of $\Lambda_\cO$ at a character of $\Gamma$ as before.

For all integers $n\ge1$, we write  $\omega_n=(1+X)^{p^n}-1$ and $\Phi_n=\omega_n/\omega_{n-1}$ (with $\omega_0=X$). We let $\Lambda_n$ denote the quotient $\Lambda/(\omega_n)=\Lambda_{\Gamma_n}$. Furthermore, we fix a primitive $p^n$-th root of unity $\zeta_{p^n}$ and write $\epsilon_n=\zeta_{p^n}-1$. We shall also assume that the primitive roots of unity are chosen such that $\zeta_{p^{n+1}}^p=\zeta_{p^{n}}$.

\section{Multiply signed Selmer groups}\label{S:signed}

Throughout this section, we fix a prime $v\in\Sigma'_p$ and a prime $w\in\Sigma_p$ lying above $v$. By (S1), we may identify $K_v$ with $\Qp$ and $F_w$ with a finite unramified extension of $\Qp$. Let  $\cO_w$ denote the ring of integers of $F_w$.

\subsection{Coleman maps at supersingular primes}
In this subsection, we shall further assume that $w\in \Sss$. 
 We then write $\Dcrisv$ and $\Dcrisw$ for the Dieudonn\'e modules of $T|_{G_{K_v}}$ and $T|_{G_{F_w}}$ respectively. Recall that $\Dcrisv$ is a filtered $\Zp$-module of rank 2 and $\Dcrisv\otimes_{\Zp}\Qp$ is equipped with a linear operator $\vp$. Furthermore, $\Dcrisw=\cO_w\otimes_{\Zp}\Dcrisv$ and $\vp$ acts semi-linearly on $\Dcrisw\otimes_{\Zp}\Qp$. That is $\vp( x\otimes u)= x^{\sigma_w}\otimes\vp(u)$ for $x\in \cO_w$ and $u\in \Dcrisv\otimes_{\Zp}\Qp$, where $\sigma_w$ is the Frobenius of $F_w/F'_v$. Let $a_v=1+p-|\tilde E_v(k_v)|\in p\ZZ$, where $\tilde E_v$ is the reduced curve of $E$ modulo $v$ and $k_v$ is the residue field of $K_v$. The characteristic polynomial of $\vp$ on $\Dcrisv\otimes\Qp$ is given by $X^2-\frac{a_v}{p}X+\frac{1}{p}$. By the theory of Fontaine-Laffaille, $\Dcrisv$ admits a $\Zp$-basis of the form $\{\omega_v,\vp(\omega_v)\}$, where $\omega_v$ generates $\Fil^0\Dcrisv$. The matrix of $\vp$ with respect to this basis is of the form
 \[
 A_v:=\begin{pmatrix}0&\frac{-1}{p}\\1&\frac{a_v}{p} \end{pmatrix}=\begin{pmatrix}0&-1\\1&a_v\end{pmatrix}\begin{pmatrix}1&0\\0&\frac{1}{p}\end{pmatrix}.
 \]
 
 For $n\ge1$, let 
 \begin{equation}
     \label{eq:defnC} C_{v,n}=\begin{pmatrix}1&0\\0&\Phi_n\end{pmatrix}\begin{pmatrix}a_v&1\\-1&0\end{pmatrix}\quad\text{and}\quad M_{v,n}=A_v^{n+1}C_{v,n}\cdots C_{v,1}.
 \end{equation}
 
 By \cite[Lemma~4.4]{sprung17} (see also \cite[Theorem~1.5]{lei15}), the matrices $M_{v,n}$ converge to a $2\times2$ matrix over $\cH$ as $n\rightarrow\infty$. We then define
 \[
 \Mlogv:=\lim_{n\rightarrow \infty}M_{v,n}.
 \]
 
  Let $\HIw(K_{\infty,v},T)$ denote the inverse limit $\varprojlim H^1(K_{n,v},T)$, where the connecting maps are corestrictions. Let
  $ \cL_v:\HIw(K_{\infty,v},T)\rightarrow \cH\otimes \Dcrisv$ be the Perrin-Riou map as given by \cite[\S5.1]{BLLV} (originally defined in \cite{perrinriou94}). The matrix $\Mlogv$ allows us to factorize $\cL_v$ into
  \begin{equation}
  \cL_v=\begin{pmatrix}\omega_v&\vp(\omega_v)\end{pmatrix}\Mlogv\begin{pmatrix}\colsv\\\colfv\end{pmatrix},
      \label{eq:decomposePR}
  \end{equation}
  where $\colsv,\colfv:\HIw(K_{\infty,v},T)\rightarrow \Lambda$ are $\Lambda$-morphisms  as given in \cite[\S2.3]{BL19}. We would like to describe the images of the Coleman maps. As a start, we recall the following preliminary lemma due to Kobayashi (see \cite[proof of Proposition~8.23]{kobayashi03}).
  
  \begin{lemma}\label{lem:surjcor}
  The corestriction map $H^1(K_{m,v},T)\rightarrow H^1(K_{n,v},T)$ is surjective for all $m\ge n$.
  \end{lemma}
  \begin{proof}
  It is well-known that in this supersingular setting, one has $H^0(K_{m,v},E[p^\infty])=0$ (cf. \cite[Proposition 8.7]{kobayashi03} or \cite[Proposition 3.1]{KitaOts}). From this fact, we then see that the restriction map $$H^1(K_{n,v},E[p^\infty])\rightarrow H^1(K_{m,v},E[p^\infty])$$ is injective. The required conclusion now follows from this and the local Tate duality.
  \end{proof}
  
  \begin{proposition}\label{prop:image}
  Let $I_v:=\{(G_1,G_2)\in\Lambda^{\oplus2}:(p-1)G_1(0)=(2-a_v)G_2(0)\}$. Then
  \[
  \image(\colsv\oplus\colfv)= I_v.
  \]
  \end{proposition}
  \begin{proof}
  By \cite[Corollary~5.3 and Theorem~5.10]{leiloefflerzerbes11}, we have the inclusion
  \[
  \image(\colsv\oplus\colfv)\subset I_v,
  \]
  with finite index.  By Nakayama's Lemma and the surjectivity of the corestriction maps as given by Lemma~\ref{lem:surjcor}, it is enough to show that 
  \begin{equation}\label{eq:image}
      \image(\colsv\oplus\colfv)\mod X=\{(g_1,g_2)\in\Zp^{\oplus2}:(p-1)g_1=(2-a_v)g_2\}.
  \end{equation}
  
  Let us recall that 
  \begin{equation}\label{eq:PRmod}
        \cL_v\equiv (1-\vp)(1-p^{-1}\vp^{-1})^{-1}\exp^*\mod X
  \end{equation}
  (see \cite[Theorem~B.5]{loefflerzerbes11} for example).
   Thus, on combining \eqref{eq:decomposePR} and \eqref{eq:PRmod}, we have
  \[
  \exp^*\equiv\frac{1}{1+p-a_v}\begin{pmatrix}
  \omega_v& \vp(\omega_v)
  \end{pmatrix}\begin{pmatrix}
  \frac{a_v-2p}{p}&\frac{1-p}{p}\\
  p-1&a_v-2
  \end{pmatrix} \begin{pmatrix}
  \colsv\\ \colfv
  \end{pmatrix}\mod X
  \] 
  (see \cite[proof of Proposition 2.12]{hatley}).
  This in turn implies that
  \begin{equation}\label{eq:relation}
      (p-1)\colsv=(2-a_v)\colfv\mod X
  \end{equation}
   and
  \begin{equation}\label{eq:colexp}
       \exp^*=\omega_v\frac{(2-a_v)(a_v-2p)-(p-1)^2}{p(p-1)(1+p-a_v)}\colsv,
  \end{equation}
  where we note that  $\frac{(2-a_v)(a_v-2p)-(p-1)^2}{p(p-1)(1+p-a_v)}\in \frac{1}{p}\Zp^\times$.
  
 {It follows from \cite[Theorem~4.1(iii)]{blochkato} that 
   \[
  \left[\exp\left(\Dcrisv/\Fil^0\Dcrisv\right) :H^1_f(K_v,T)\right]=|\det(1-\vp)|_p=p,
  \]
  where $\exp$ is the Bloch-Kato exponential map and $|\cdot|_p$ is the $p$-adic norm normalized by $|p|_p=p^{-1}$.
  It then follows that the dual exponential map gives an isomorphism
  \begin{equation}\label{eq:dualexp}
        \exp^*:H^1_{/f}(K_v,T)\stackrel{\sim}{\longrightarrow}\frac{1}{p}\Zp\omega_v.
  \end{equation}}
  Combining \eqref{eq:colexp} and \eqref{eq:dualexp} yields $$\image(\colfv)_\Gamma=\Zp.$$
  This,  together with \eqref{eq:relation} allow us to deduce \eqref{eq:image}.
  \end{proof}

  After tensoring by $\cO_w$, we have similarly
   \[
  \cL_w=\begin{pmatrix}\omega_v&\vp(\omega_v)\end{pmatrix}\Mlogv\begin{pmatrix}\colsw\\\colfw\end{pmatrix},
  \]
  where $\cL_w$ is the Perrin-Riou map from $\HIw(F_{\infty,w},T)=\HIw(K_{\infty,v},T)\otimes\cO_w$ to $\cH\otimes\Dcrisw$ and $\colsw$ and $\colfw$ are defined by extending $\colsv$ and $\colfv$ $\cO_w$-linearly.

 \begin{remark}
  We have automatically $ \image(\colsw\oplus\colfw)=I_v\otimes\cO_w$.
  \end{remark}

\subsection{Selmer groups}

Let $w\in\Sss$. Consider the local Tate pairing
\[
\HIw(F_{\infty,w},T)\times H^1(F_{\infty,w},\Ep)\rightarrow \Qp/\Zp.
\]
For $\bullet\in\{\sharp,\flat\}$, we define $H^1_\bullet(F_{\infty,w},\Ep)\subset  H^1(F_{\infty,w},\Ep)$ to be the orthogonal complement of  $\ker\colbw$ under the local Tate pairing. We shall write
\[
H^1_{/\bullet}(F_{\infty,w},\Ep)=\frac{H^1(F_{\infty,w},\Ep)}{H^1_\bullet(F_{\infty,w},\Ep)}.
\]
Similarly, if $v\in\Sigma_p$, we write 
\[
H^1_{/f}(F_{\infty,v},\Ep)=\frac{H^1(F_{\infty,v},\Ep)}{E(F_\infty,w)\otimes\Qp/\Zp}, 
\]
where $E(F_\infty,w)\otimes\Qp/\Zp$ is identified with its image inside $H^1(F_{\infty,w},\Ep)$ under the Kummer map.

Let $\fs=(s_w)_{w\in\Sss}\in\{\flat,\sharp\}^\Sss$, we define the signed Selmer group of $E$ over $F_\infty$ by
\begin{align*}
    \Sel^{\fs}(E/F_\infty):=\ker\Big(H^1(F_\infty,\Ep)\rightarrow&\\ \prod_{w\in\Sss}H^1_{/s_w}(F_{\infty,w},\Ep)\times&\prod_{w\in\Sord}H^1_{/f}(F_{\infty,w},\Ep)\times \prod_w H^1(F_{\infty,w},\Ep)  \Big),
\end{align*}
where the last product runs through all primes of $F_\infty$ not dividing $p$. Equivalently, if $\Selp(E/F_\infty)$ denotes the classical $p^\infty$-Selmer group, then
\[
\Sel^{\fs}(E/F_\infty)=\ker\Big(\Selp(E/F_\infty)\rightarrow \prod_{w\in\Sss}H^1_{/s_w}(F_{\infty,w},\Ep)\Big).
\]
It is well-known that $\Selp(E/F_\infty)$ is cofinitely generated over $\Lambda$ (cf. \cite[Theorem 4.5]{Manin}). Thus, so is $\Sel^{\fs}(E/F_\infty)$.

\begin{conjecture}\label{conj:tor}
For all choices of $\fs$, the Selmer group $\Sel^{\fs}(E/F_\infty)$ is cotorsion over $\Lambda$.
\end{conjecture}

For the rest of the article, we assume that the following hypothesis holds:

\begin{itemize}
    \item[(S3)] Conjecture~\ref{conj:tor} holds. 
\end{itemize}
\begin{defn}
We write  $\mu_\fs$ and $\lambda_\fs$ for the $\mu$- and $\lambda$-invariants of the torsion $\Lambda$-module $\Sel^{\fs}(E/F_\infty)^\vee$.
\end{defn}

When the elliptic curve $E$ has good ordinary reduction at all primes above $p$, the above conjecture is precisely Mazur's conjecture \cite{mazur72} which is known to be valid in the case when $E$ is defined over $\QQ$ and $F$ an abelian extension of $\QQ$ (see \cite{kato04}). For an elliptic curve over $\QQ$ with good supersingular reduction at $p$, this conjecture was established by Kobayashi
(cf.\ \cite{kobayashi03}; also see \cite{BL14} for some recent progress on this conjecture).

\subsection{Structures of global cohomologies}

In this section, we record certain consequences of Conjecture~\ref{conj:tor}, which will be utilized in subsequent sections of the paper. From now on, let $\Sigma$ denote a fixed finite set of primes of $F$ containing those above $p$, the ramified primes of $F/K$ and all the bad reduction primes of $E$. Write $F_{\Sigma}$ for the maximal algebraic extension of $F$ which is unramified outside $\Sigma$. For any (possibly infinite) extension $F\subseteq L\subseteq F_{\Sigma}$, write $G_{\Sigma}(L)= \Gal(F_{\Sigma}/L)$.
The signed Selmer group of $E$ over $F_\infty$ can then be equivalently defined by
\begin{align*}
    \Sel^{\fs}(E/F_\infty):=\ker\Big(H^1(G_{\Sigma}(F_\infty),\Ep)\rightarrow&\\ \prod_{w\in\Sss}H^1_{/s_w}(F_{\infty,w},\Ep)\times&\prod_{w\in\Sord}H^1_{/f}(F_{\infty,w},\Ep)\times \prod_{w\in\Sigma, w\nmid p} H^1(F_{\infty,w},\Ep)  \Big).
\end{align*}

We also define $H^i_{\Iw,\Sigma}(F_{\infty}, T)=\mathop{\varprojlim}\limits_n H^i(G_{\Sigma}(F_{n}), T)$, where the transition maps are given by the corestriction maps. {Note that $H^1_{\Iw,\Sigma}(F_{\infty}, T)$ is independent of the choice of $\Sigma$ (see \cite[Lemma~5.3.1]{MR} or \cite[Proposition 7.1]{kobayashi03}). Since our set $\Sigma$ is  fixed throughout, we will drop the subscript $\Sigma$ from the notation for simplicity and write $H^i_{\Iw}(F_{\infty}, T)$.} We now record the following useful observation.

\begin{lemma} \label{H1Iw} The group $\HIw(F_{\infty}, T)$ is a torsion-free $\Lambda$-module. In the event that $(S1)$ and $(S2)$ are valid, we even have that $H^1(G_{\Sigma}(F), T)$ is a torsion-free $\Zp$-module. 
\end{lemma}

\begin{proof}
By considering the low degree terms of the spectral sequence of Jannsen
\[ \Ext^i_{\Lambda}\big(H^j(G_{\Sigma}(F_{\infty}),\Ep)^{\vee},\Lambda\big) \Longrightarrow H_{\mathrm{Iw}}^{i+j}(F_{\infty}, T)\]
(cf. \cite[Theorem 1]{Jannsen}), we obtain the following exact sequence
\[ 0\lra \Ext^1_{\Lambda}((E(F_{\infty})[p^{\infty}])^{\vee},\Lambda\big) \lra \HIw(F_{\infty}, T) \lra \Ext^0_{\Lambda}\big(H^1(G_{\Sigma}(F_{\infty}),\Ep)^{\vee},\Lambda\big). \]
By a theorem of Imai \cite{Imai}, $E(F_{\infty})[p^{\infty}]$ is finite and so the leftmost term vanishes. This in turn implies that $\HIw(F_{\infty}, T)$ injects into an $\Ext^0$-term. Since the latter is a reflexive $\Lambda$-module by \cite[Corollary 5.1.3]{NSW}, $\HIw(F_{\infty}, T)$ must be torsionfree. 

We now prove the second assertion. The low degree terms of the spectral sequence 
\[ \Ext^i_{\Zp}\big(H^j(G_{\Sigma}(F),\Ep)^{\vee},\Zp\big) \Longrightarrow H^{i+j}(G_{\Sigma}(F), T)\] yields the following exact sequence
\[ 0\lra \Ext^1_{\Zp}((E(F)[p^{\infty}])^{\vee},\Zp\big) \lra H^{1}(G_{\Sigma}(F), T) \lra \Ext^0_{\Zp}\big(H^1(G_{\Sigma}(F_{\infty}),\Ep)^{\vee},\Zp\big). \]
Since $(S1)$ and $(S2)$ are valid, the proof of Lemma~\ref{lem:surjcor} tells us that  $E(F_{w})[p^{\infty}]=0$ for $w\in\Sss$. From which, one has $E(F)[p^{\infty}]=0$. Consequently, we have $H^{1}(G_{\Sigma}(F), T)$ injecting into an $\Ext^0$-term and so it must be $\Zp$-torsionfree.
\end{proof}

\begin{remark}
It is clear from the proof of Lemma \ref{H1Iw} that under the validity of $(S1)$ and $(S2)$, we can also show that $H^1(G_{\Sigma}(F_n), T)$ is a torsion-free $\Zp$-module for every $n$.
\end{remark}

\begin{proposition} \label{torsion H2}
Suppose that $(S1)$ and $(S2)$ are valid. Then $\Sel^{\fs}(E/F_\infty)$ is a cotorsion $\Lambda$-module if and only if we have that $H^2(G_{\Sigma}(F_{\infty}),\Ep)=0$ 
 and that the following  sequence 
 \begin{align*} 
 0\lra \Sel^{\fs}(E/F_\infty) \lra H^1(G_{\Sigma}(F_{\infty}),\Ep)\lra &\\ \prod_{w\in\Sss}H^1_{/s_w}(F_{\infty,w},\Ep)\times &\prod_{w\in\Sord}H^1_{/f}(F_{\infty,w},\Ep)\times \prod_{w\in\Sigma, w\nmid p} H^1(F_{\infty,w},\Ep) \lra 0.  
 \end{align*}
 is exact.
\end{proposition}

\begin{proof}
To simplify notation, we write $J_w(E/F_{\infty})$ for each of the local summands.
By \cite[Proposition~A.3.2]{perrinriou95}, we have an exact sequence
\[0\lra \Sel^{\overrightarrow{s}}(E/F_{\infty})\lra H^1(G_{\Sigma}(F_{\infty}),\Ep)\lra \prod_{w\in\Sigma}J_w(E/F_{\infty})  \]
\[ \lra \mathfrak{S}^{\overrightarrow{s}}(E/F_{\infty})^{\vee}\lra H^2(G_{\Sigma}(F_{\infty}),\Ep)\lra 0,\]
where $\mathfrak{S}^{\overrightarrow{s}}(E/F_{\infty})$ is a $\Lambda$-submodule of $\HIw(F_{\infty}/F, T)$. (For the precise definition of $\mathfrak{S}^{\overrightarrow{s}}(E/F_{\infty})$, we refer readers to loc. cit. For our purposes, the submodule theoretical information suffices.) Standard corank calculations \cite[Propositions 1-3]{G89} and \cite[Proposition 3.32]{KitaOts} tell us that
\[ \mathrm{corank}_{\Lambda}\big(H^1(G_{\Sigma}(F_{\infty}),\Ep)\big)- \mathrm{corank}_{\Lambda}\big(H^2(G_{\Sigma}(F_{\infty}),\Ep)\big) =[F: \QQ], \]
  and
\[\mathrm{corank}_{\Lambda}\left(\bigoplus_{w\in \Sigma}J_w(E/F_{\infty}) \right) = [F:\QQ].\]
It is now clear from these formulas and the above exact sequence that $\Sel^{\overrightarrow{s}}(E/F_{\infty})$ is a cotorsion $\Lambda$-module if and only if $\mathfrak{S}^{\overrightarrow{s}}(E/F_{\infty})$ is a torsion $\Lambda$-module. Since $\mathfrak{S}^{\overrightarrow{s}}(E/F_{\infty})$ is contained in $\HIw(F_{\infty}, T)$ which is torsionfree by Lemma \ref{H1Iw}, the latter statement holds if and only if $\mathfrak{S}^{\overrightarrow{s}}(E/F_{\infty})=0$. But this is precisely equivalent to having
$H^2(G_{\Sigma}(F_{\infty}),\Ep)=0$ and the short exact sequence in the proposition.
\end{proof}

The next proposition records consequence of the cotorsionness of $\Sel^{\fs}(E/F_{\infty})$ on the structure of the Iwasawa cohomology groups $H^i_{\Iw}(F_{\infty},T)$.

\begin{proposition}\label{prop:leopoldt}
Assume that $(S1)$ and $(S2)$ are valid. Suppose that $\Sel^{\fs}(E/F_{\infty})$ is cotorsion over $\Lambda$. Then the following statements are valid.
\begin{itemize}
    \item[(a)] $\HIw(F_{\infty}, T)$ is a 
free $\Lambda$-module with $\Lambda$-rank $[F:\QQ]$.

\item[(b)] $H^2_{\Iw}(F_{\infty}, T)$ is a 
torsion $\Lambda$-module.
\end{itemize}
\end{proposition}

\begin{proof}
It follows from the hypothesis and Proposition \ref{torsion H2} that $H^2(G_{\Sigma}(F_{\infty}),\Ep)=0$. { Also, as seen in the proof of Lemma~\ref{H1Iw}, we have $E(F_{\infty})[p^\infty]=0$.  Taking these into account}, the spectral sequence of Jannsen
\[ \Ext^i_{\Lambda}\big(H^j(G_{\Sigma}(F_{\infty}),\Ep)^{\vee},\Lambda\big) \Longrightarrow H_{\mathrm{Iw}}^{i+j}(F_{\infty}, T)\]
then degenerates yielding
\[\HIw(F_{\infty},T)\cong \Ext^0(H^1(G_{\Sigma}(F_{\infty}),\Ep)^{\vee},\Lambda)\]
and
\[H^2_{\Iw}(F_{\infty},T)\cong \Ext^1(H^1(G_{\Sigma}(F_{\infty}),\Ep)^{\vee},\Lambda).\]
It follows from this that $\HIw(F_{\infty}, T)$ has $\Lambda$-rank $[F:\QQ]$ and $H^2_{\Iw}(F_{\infty}, T)$ is a 
torsion $\Lambda$-module.
On the other hand, the homological spectral sequence 
\[ H^i(\Gamma, H^{-j}_{\Iw}(F_{\infty}, T))\Longrightarrow H^{-i-j}(G_{\Sigma}(F), T)\]
(cf. \cite[Theorem 3.1.8]{LimSharifi}) gives an injection 
\[ \HIw(F_{\infty},T)_{\Gamma}\hookrightarrow H^1(G_{\Sigma}(F),T).  \]
Since $H^1(G_{\Sigma}(F),T)$ is $\Zp$-torsionfree by Lemma \ref{H1Iw}, so is $\HIw(F_{\infty},T)_{\Gamma}$. On the other hand, by Lemma \ref{H1Iw} again, $\HIw(F_{\infty},T)$ is $\Lambda$-torsionfree and hence it follows that $\HIw(F_{\infty},T)^{\Gamma} =0$. We may now apply \cite[Proposition 5.3.19(ii)]{NSW} to obtain the  freeness property.
\end{proof}

\section{Kobayashi ranks}\label{S:kob}
Following \cite[\S10]{kobayashi03}, we define the Kobayashi ranks as follows.
\begin{defn}
Let $(M_n)_{n\ge1}$ be a projective system of finitely generated $\Zp$-modules with the connecting maps $\pi_n:M_n\rightarrow M_{n-1}$. If the kernel and the cokernel of $\pi_n$ are both finite, we define
    \[
     \nabla M_n=\length_{\Zp}\ker\pi_n-\length_{\Zp}\coker\pi_n+\dim_{\Qp}M_{n-1}\otimes_{\Zp}\Qp.
    \]
\end{defn}

{ 
We mention two important observations which will be used in the subsequent sections of the paper. Firstly, for a given projective system  $(M_n)_{n\ge1}$ of finitely generated $\Zp$-modules, 
$\rank_{\Zp}M_n$ is bounded independent of $n$ if and only if $\nabla M_n$ is defined for $n\gg0$. Secondly, in the event that  $(M_n)_{n\ge1}$ is a projective system of finite $\Zp$-modules, we have
\[ \nabla M_n= e(M_n)- e(M_{n-1}). \]}

We now review further important properties of Kobayashi ranks that we shall need later.

\begin{proposition}\label{prop:nabla}Let $n\ge1$ be an integer.
\begin{itemize}
    \item[(a)] Let $f\in\Lambda$. Suppose that $\Phi_n\nmid f$ and write $$f=gh,\quad\text{where }h=\gcd(\omega_{n-1},f).$$
Then 
$$\dim_{\Qp}\left(\Lambda/(f,\omega_{n-1})\right)\otimes\Qp=\ord_{\epsilon_n}h(\epsilon_n)$$
Furthermore, $\pi:\Lambda/(f,\omega_n)\rightarrow \Lambda/(f,\omega_{n-1})$ has finite kernel with $$\length_{\Zp}\ker\pi=\ord_{\epsilon_n}g(\epsilon_n).$$
In particular,
\[
\nabla \Lambda/(f,\omega_n)=\ord_{\epsilon_n}g(\epsilon_n)+\ord_{\epsilon_n}h(\epsilon_n)=\ord_{\epsilon_n}f(\epsilon_n).
\]
\item[(b)] Suppose that $M$ is a finitely generated torsion $\Lambda$-module and that $f$ is a characteristic element of $M$. Then $\nabla M_{\Gamma_n}$ is defined and equals
\[
\ord_{\epsilon_n}f(\epsilon_n)=p^{n-1}(p-1)\mu(M)+\lambda(M)
\]
when $n\gg0$.
\end{itemize}
 \end{proposition}
\begin{proof}
This is \cite[Lemma~10.5]{kobayashi03}.
\end{proof}

We shall need a slightly more general version of part (a) of this proposition, which can be proved in exactly the same way.
\begin{proposition}\label{prop:nablawithcoeff}Let $n\ge1$ be an integer and $\cO$ the ring of integers of a finite extension of $\Qp$ of degree $k$.
 Let $f\in\Lambda_\cO$. Suppose that $\Phi_n\nmid f$ and write $$f=gh,\quad\text{where }h=\gcd(\omega_{n-1},f).$$
Then 
$$\dim_{\Qp}\left(\Lambda_\cO/(f,\omega_{n-1})\right)\otimes\Qp=k\cdot\ord_{\epsilon_n}h(\epsilon_n)$$
Furthermore, $\pi:\Lambda_\cO/(f,\omega_n)\rightarrow \Lambda_\cO/(f,\omega_{n-1})$ has finite kernel with $$\length_{\Zp}\ker\pi=k\cdot \ord_{\epsilon_n}g(\epsilon_n).$$
In particular,
\[
\nabla \Lambda_\cO/(f,\omega_n)=k\cdot\ord_{\epsilon_n}f(\epsilon_n).
\]
\end{proposition}

We shall also need the following lemma on how Kobayashi ranks behave under short exact sequences.

\begin{lemma}\label{lem:kobSES}
Suppose we have a short exact sequence of inverse systems
\[
0\rightarrow (M_n')\rightarrow (M_n)\rightarrow (M_n'')\rightarrow 0.
\]
If two of $\nabla M_n'$, $\nabla M_n$, $\nabla M_n''$ are defined, then so is the third. Furthermore,
\[
\nabla M_n=\nabla M_n'+\nabla M_n''.
\]
\end{lemma}
\begin{proof}
This is \cite[Lemma~10.4 i)]{kobayashi03}.
\end{proof}

We record the following lemma, which will be useful later.

\begin{lemma}\label{lemma:NSWfg}
Let $M$ be a finitely generated $\Lambda$-module. Then $M^{\delta} : = \cup_{n}M^{\Gamma_n}$ is finitely generated as a $\Zp$-module and there exists an integer $n_0$ such that $M^{\delta}  = M^{\Gamma_{n_0}}$.

In particular, $\nabla M^{\Gamma_n} =0$ for $n\gg 0$, where the transition maps $ M^{\Gamma_{n+1}}\lra  M^{\Gamma_n}$ are given by multiplication by $1+\gamma_n+\cdots +\gamma_n^p$ with $\gamma_n$ being a topological generator of $\Gamma_n$ chosen so that $\gamma_n^p=\gamma_{n+1}$.
\end{lemma}

\begin{proof}
The first assertion of the lemma follows from \cite[Lemma 5.3.14(i)]{NSW}. As a result, there exists $n_0$ such that $M^{\Gamma_{n}} = M^{\Gamma_{n_0}}$ is finitely generated over $\Zp$ for $n\geq n_0$. It then follows that the transition map $M^{\Gamma_{n+1}}\lra  M^{\Gamma_n}$ coincides with the multiplication by $p$-map for $n\gg 0$. The lemma now follows from \cite[Lemma~10.4(ii)]{kobayashi03}.
\end{proof}

Similar to \cite[\S10]{kobayashi03}, we consider the following groups
 \[\mathcal{Y}(E/F_n)= \coker\left(H^1(G_{\Sigma}(F_n),T)\longrightarrow\bigoplus_{w\in\Sigma_p}H^1_{/f}(F_{n,w},T) \right),\]
 \[\cY'(E/F_n) = \coker\left(\HIw(F_\infty,T)_{\Gamma_n}\longrightarrow\bigoplus_{w\in\Sigma_p}H^1_{/f}(F_{n,w},T) \right),\]
 where $H^1_{/f}(F_{n,w},T)$ denotes the quotient $\frac{H^1(F_{n,w},T)}{E(F_{n,w})\otimes\Zp}$.
As we shall see in \S\ref{sec:proofs}, one of the key ingredients of studying the growth of  $\rank_{\ZZ}E(F_n)$ and $\sha_p(E/F_n)$ is to understand $\nabla\cY(E/F_n)$. We end this section by the following generalization of \cite[Proposition~10.6 i)]{kobayashi03}.

\begin{proposition}\label{prop:equivY}
Suppose that $(S1)$ and $(S2)$ are valid. Then for $n\gg 0$, we have \[\nabla \mathcal{Y}(E/F_n) = \nabla \mathcal{Y'}(E/F_n).\]
\end{proposition}

\begin{proof}
For each $n$, we have the following commutative diagram
\[  \entrymodifiers={!! <0pt, .8ex>+} \SelectTips{eu}{}\xymatrix{
    H^i(G_{\Sigma}(F), \Zp[\Gal(F_{n+1}/F)]\otimes_{\Zp} T) \ar[d]_{\pr} \ar[r]^(0.62){sh}_(0.62){\sim} & H^i(G_{\Sigma}(F_{n+1}),T) \ar[d]_{\cor}  \\
   H^i(G_{\Sigma}(F), \Zp[\Gal(F_n/F)]\otimes_{\Zp} T)  \ar[r]^(0.6){sh}_(0.6){\sim} &  H^i(G_{\Sigma}(F_{n}),T), } \]
which upon taking inverse limits induces an isomorphism
\[ H^i(G_{\Sigma}(F), \LL^{\iota}\otimes T) \cong \plim_n H^i(G_{\Sigma}(F_{n+1}),T)=: H^i_{\Iw}(F_{\infty}, T),  \]
{where here $\LL^{\iota}$ is $\LL$ as a $\Zp$-module  on which $\gamma\in\Gamma$ acts via multiplication by $\gamma^{-1}$. }

Write $\gamma_n$ for a topological generator of $\Gamma_n$ which is chosen so that $\gamma_n^p=\gamma_{n+1}$.
Taking $G_\Sigma(F)$-cohomology in the short exact sequence
\[ 0\lra \LL^{\iota}\otimes_{\Zp} T \stackrel{\gamma_n-1}{\lra} \LL^{\iota}\otimes_{\Zp} T \lra \Zp[\Gal(F_n/F)]\otimes_{\Zp} T\lra 0\]
yields the following long exact sequence
\[ H^1_{\Iw}(F_{\infty}, T)  \stackrel{\gamma_n-1}{\lra} H^1_{\Iw}(F_{\infty}, T) \lra H^1(G_{\Sigma}(F_{n}),T) \lra H^2_{\Iw}(F_{\infty}, T)  \stackrel{\gamma_n-1}{\lra}  H^2_{\Iw}(F_{\infty}, T) ,\]
which in turn yields the following short exact sequence
\[ 0\lra H^1_{\Iw}(F_{\infty}, T)_{\Gamma_n}\lra H^1(G_{\Sigma}(F_{n}),T) \lra H^2_{\Iw}(F_{\infty}, T)^{\Gamma_n}\lra 0.\]

Furthermore,  the commutative diagram
\[   \entrymodifiers={!! <0pt, .8ex>+} \SelectTips{eu}{}\xymatrix{
    0 \ar[r]^{} & \LL^{\iota}\otimes_{\Zp} T \ar[d]_{1+\gamma_{n}+\cdots +\gamma_{n}^{p-1}} \ar[r]^{\gamma_{n+1}-1} &  \LL^{\iota}\otimes_{\Zp} T
    \ar@{=}[d]^{} \ar[r]^(){} & \Zp[\Gal(F_{n+1}/F)]\otimes_{\Zp} T\ar[d]^{\pr} \ar[r] &0\\
    0 \ar[r]^{} & \LL^{\iota}\otimes_{\Zp} T \ar[r]^{^{\gamma_{n}-1}} & 
    \LL^{\iota}\otimes_{\Zp}T \ar[r] & \Zp[\Gal(F_{n}/F)]\otimes_{\Zp} T \ar[r] &0 } \]
 induces the commutative diagram
\[   \xymatrixrowsep{0.25in}
\xymatrixcolsep{0.15in}\entrymodifiers={!! <0pt, .8ex>+} \SelectTips{eu}{}\xymatrix{
    0 \ar[r]^{} & H^1_{\Iw}(F_{\infty}, T)_{\Gamma_{n+1}} \ar[d] \ar[r] &  H^1(G_{\Sigma}(F_{n+1}),T)
    \ar[d]^{\cor} \ar[r]^(){} & H^2_{\Iw}(F_{\infty}, T)^{\Gamma_{n+1}}\ar[d]^{1+\gamma_{n}+\cdots +\gamma_{n}^{p-1}} \ar[r]^{} &0\\
    0 \ar[r]^{} & H^1_{\Iw}(F_{\infty}, T)_{\Gamma_{n}}\ar[r]^{} & H^1(G_{\Sigma}(F_{n}),T) \ar[r] &H^2_{\Iw}(F_{\infty}, T)^{\Gamma_{n}} \ar[r] &0 } \]
    with exact rows. 
    
For each $n$, we also have the following commutative diagram
\[  \entrymodifiers={!! <0pt, .8ex>+} \SelectTips{eu}{}\xymatrix{
    H^i(G_{\Sigma}(F), \Zp[\Gal(F_{n}/F)]\otimes_{\Zp} T) \ar[d]_{1+\gamma_{n}+\cdots +\gamma_{n}^{p-1}} \ar[r]^(0.62){sh}_(0.62){\sim} & H^i(G_{\Sigma}(F_{n}),T) \ar[d]_{\res}  \\
   H^i(G_{\Sigma}(F), \Zp[\Gal(F_{n+1}/F)]\otimes_{\Zp} T)  \ar[r]^(0.6){sh}_(0.6){\sim} &  H^i(G_{\Sigma}(F_{n+1}),T) } \]
  and the following commutative diagram
\[   \entrymodifiers={!! <0pt, .8ex>+} \SelectTips{eu}{}\xymatrix{
    0 \ar[r]^{} & \LL^{\iota}\otimes_{\Zp} T \ar@{=}[d] \ar[r]^{\gamma_{n}-1} & 
    \LL^{\iota}\otimes_{\Zp} T
    \ar[d]^{1+\gamma_{n}+\cdots +\gamma_{n}^{p-1}} \ar[r]^(){} & \Zp[\Gal(F_{n}/F)]\otimes_{\Zp} T\ar[d]^{1+\gamma_{n}+\cdots +\gamma_{n}^{p-1}} \ar[r] &0\\
    0 \ar[r]^{} & \LL^{\iota}\otimes_{\Zp} T \ar[r]^{^{\gamma_{n+1}-1}} & 
    \LL^{\iota}\otimes_{\Zp}T \ar[r] & \Zp[\Gal(F_{n+1}/F)]\otimes_{\Zp} T \ar[r] &0. } \]
This gives     the commutative diagram
\[   \xymatrixrowsep{0.25in}
\xymatrixcolsep{0.15in}\entrymodifiers={!! <0pt, .8ex>+} \SelectTips{eu}{}\xymatrix{
    0 \ar[r]^{} & H^1_{\Iw}(F_{\infty}, T)_{\Gamma_{n}} \ar[d]^{1+\gamma_{n}+\cdots +\gamma_{n}^{p-1}} \ar[r] &  H^1(G_{\Sigma}(F_{n}),T)
    \ar[d]^{\res} \ar[r]^(){} & H^2_{\Iw}(F_{\infty}, T)^{\Gamma_{n}}\ar[d]^{\subseteq} \ar[r]^{} &0\\
    0 \ar[r]^{} & H^1_{\Iw}(F_{\infty}, T)_{\Gamma_{n+1}}\ar[r]^{} & H^1(G_{\Sigma}(F_{n+1}),T) \ar[r] &H^2_{\Iw}(F_{\infty}, T)^{\Gamma_{n+1}} \ar[r] &0 } \]
with exact rows.

Applying the snake lemma to the following commutative diagram
\[   \xymatrixrowsep{0.25in}
\xymatrixcolsep{0.15in}\entrymodifiers={!! <0pt, .8ex>+} \SelectTips{eu}{}\xymatrix{
    0 \ar[r]^{} & H^1_{\Iw}(F_{\infty}, T)_{\Gamma_{n}} \ar[d] \ar[r] &  H^1(G_{\Sigma}(F_{n}),T)
    \ar[d]^{} \ar[r]^(){} & H^2_{\Iw}(F_{\infty}, T)^{\Gamma_{n}} \ar[r] &0\\
    & \bigoplus_{v|p}H^1_{/f}(F_{n,w},T) \ar@{=}[r] & \bigoplus_{v|p}H^1_{/f}(F_{n,w},T)& &  } \]
    yields
    \[H^2_{\Iw}(F_{\infty}, T)^{\Gamma_{n}}\stackrel{\partial_n}{\lra} \mathcal{Y}'(E/F_n)\lra \mathcal{Y}(E/F_n)\lra 0. \]
    Denote by $Z'_n$ and $Z_n''$ the kernel and image of $\partial_n$. It is straightforward to check that the preceding diagram is compatible in $n$, going from $n$ to $n+1$.
    By Lemma \ref{lemma:NSWfg}, $H^2_{\Iw}(F_{\infty}, T)^{\Gamma_{n}}$ stabilizes for $n\gg 0$ and so one has $Z'_n\subseteq Z_n''$ for $n\gg 0$. Since
    $H^2_{\Iw}(F_{\infty}, T)^\delta$ is finitely generated over $\Zp$, it follows from the Noetherian property that $Z_n'$ stabilizes for $n\gg 0$. Hence so does $Z_n''$. Now going from $n+1$ to $n$, it follows from Lemma \ref{lemma:NSWfg} that the transition maps on $H^2_{\Iw}(F_{\infty}, T)^{\Gamma_{n}}$ is given by multiplication of $p$ for $n\gg 0$. Therefore, the transition maps on $Z'_n$ is given by $p$ for $n\gg 0$, and hence the same can be said for $Z_n''$. We may now apply \cite[Lemma 10.4(ii)]{kobayashi03} to conclude that $\nabla Z''_n=0$ for $n\gg 0$. Consequently, we have $\nabla \mathcal{Y}'(E/F_n)= \nabla\mathcal{Y}(E/F_n)$ for $n\gg 0$ and this finishes the proof of the proposition.
    \end{proof}

\section{Local analysis via Coleman maps}\label{S:local}
As remarked in the previous section, we shall need to understand the growth of $\cY(E/F_n)$, or equivalently, the growth of $\cY'(E/F_n)$ thanks to Proposition~\ref{prop:equivY}. We shall study the image of  $\HIw(F_\infty,T)$ in the quotient $H^1_{/f}(F_{n,w},T)$, where $w\in \Sss$, via Coleman maps. We shall consider the supersingular case and the ordinary case separately.

\subsection{The supersingular case}
Throughout this section. We fix $w\in\Sss$, which lies above $v\in\Sssp$. We shall write $\sigma_w$ for the Frobenius element of $\Gal(F_w/K_v)$. For $n\ge1$, we define
\[
H_{v,n}=C_{v,n}\cdots C_{v,1},
\]
where the matrices $C_{v,i}$ are defined as in \eqref{eq:defnC}.
We write $H_{v,n}^\sharp$ and $H_{v,n}^\flat$ for the entries of the first row of the matrix $H_{v,n}$. Then we see that $H_{v,n}$ is of the form
\begin{equation}\label{eq:Hvn}
    H_{v,n}=\begin{pmatrix}
H_{v,n}^\sharp&H_{v,n}^\flat\\
-\Phi_nH_{v,n-1}^\sharp&-\Phi_nH_{v,n-1}^\flat\end{pmatrix}.
\end{equation}

\begin{lemma}
Let $z\in \HIw(F_{w},T)$ and  $\theta$ a character on $\Gamma$ of conductor $p^{n+1}>1$ (so that it factors through $\Gamma_n$ but not $\Gamma_{n-1}$). Let $e_\theta$ denote the idempotent associated to $\theta$. Then the image of $z$ in $e_\theta\cdot H^1_{/f}(F_{n,w},T)$ is zero if and only if 
\[
H_{v,n}^\sharp\col_{\sharp,w}(z)+H_{v,n}^\flat\col_{\flat,w}(z)\in\Lambda_{\cO_w}
\]
vanishes at $\theta$.
\end{lemma}
\begin{proof}
This is proven in \cite[Proposition~5.1]{LS19}.
\end{proof}

\begin{defn}
Let $I_v$ be as defined in Proposition~\ref{prop:image}. For $n\ge1$ and $u\in\Zp^\times$, we define
\begin{align*}
    \hvnu:I_v&\rightarrow\Lambda_n\\
    (G_1,G_2)&\mapsto H^\sharp_{v,n}G_1+ uH^\flat_{v,n}G_2\mod \omega_n,
\end{align*}
and
\begin{align*}
    \colvnu:\HIw(F_{\infty,v},T)\rightarrow \Lambda_n
\end{align*}
by $\hvnu\circ (\colsv,\colfv)$ (which makes sense thanks to Proposition~\ref{prop:image}). We define $$\colwnu:\HIw(F_{\infty,w},T)\rightarrow\Lambda_n\otimes\cO_w$$ similarly.
\end{defn}

Note that $\colwnu$ factors through $\HIw(F_{\infty,w},T)_{\Gamma_n}=H^1(F_{n,w},T)$ (see \cite[proof of Proposition 3.11]{LLZ17}). In fact, one can do better and this is the content of the next lemma.

\begin{lemma}\label{lem:prelim}Let $n\ge1$. We have:
\begin{itemize}
    \item[(a)] $\image\left(\hvnu\right)\supset\omega_{n-1}\Lambda_n$;
    \item[(b)] There exists $u\in\Zp^\times$ such that $\colwnu$ induces an injection
    \[H^1_{/f}(F_{n,w},T)\hookrightarrow \Lambda_n\otimes\cO_w\]
    with finite cokernel. In particular, it is an isomorphism after tensoring by $\Qp$.
\end{itemize}
\end{lemma}
\begin{proof}
We can calculate explicitly that $\det(H_{v,n})=\omega_n/X$. Thus, \eqref{eq:Hvn} tells us that
\[
-H_{v,n}^\sharp H_{v,n-1}^\flat+H_{v,n}^\flat H_{v,n-1}^\sharp=\omega_{n-1}/X.
\]
Since $(-XH_{v,n-1}^\flat, u^{-1}XH_{v,n-1}^\sharp)\in I_v$, it is mapped to $\omega_{n-1}$ under $\hvnu$. This proves (a).
Part (b) is \cite[Proposition~3.11]{LLZ17}.
\end{proof}

\begin{corollary}\label{cor:comparepi}
Let $n\ge1$, $M$  a $\Lambda$-submodule of $H^1_{/f}(F_{n,w},T)$ and $u\in\Zp^\times$ satisfying Lemma~\ref{lem:prelim}(b). Consider the following natural projections
\begin{align*}
    \pi:&\frac{H^1_{/f}(F_{n,w},T)}{M}\rightarrow \frac{H^1_{/f}(F_{n-1,w},T)}{M_{\Gamma_{n-1}}},\\
\pi':&\frac{\Lambda_n\otimes\cO_w}{\colwnu(M)}\rightarrow \frac{\Lambda_{n-1}\otimes\cO_w}{\colwnu(M)_{\Gamma_{n-1}}}.
\end{align*}
Then one has $\ker\pi\cong \ker\pi'$ as $\Zp$-modules.
\end{corollary}
\begin{proof}
We abuse notation writing $\colwnu$ for the map 
$$H^1(F_{n-1,w},T)\rightarrow \Lambda_{n-1}\otimes\cO_w$$
defined by $\colwnu \mod\omega_{n-1}$. Consider the following commutative diagram:
\[
    \xymatrix{0\ar[r] & \frac{H^1_{/f}(F_{n,w},T)}{M}\ar[r]\ar[d]^{\pi}&\frac{\Lambda_n\otimes\cO_w}{\colwnu(M)}\ar[r]\ar[d]^{\pi'}&\frac{\Lambda_n\otimes\cO_w}{\colwnu(H^1_{/f}(F_{n,w},T))}\ar[r]\ar[d]^{\pi''}&0\\
         0\ar[r] & \frac{H^1_{/f}(F_{n-1,w},T)}{M_{\Gamma_{n-1}}}\ar[r]&\frac{\Lambda_{n-1}\otimes\cO_w}{\colwnu(M)_{\Gamma_{n-1}}}\ar[r]&\frac{\Lambda_{n-1}\otimes\cO_w}{\colwnu(H^1_{/f}(F_{n-1,w},T))}\ar[r]&0}.
\]
Recall from Lemma~\ref{lem:surjcor} that $\pi$ is surjective.
The snake lemma then gives the following short exact sequence:
\[
0\rightarrow\ker\pi\rightarrow\ker\pi'\rightarrow\ker\pi''\rightarrow 0.
\]
Therefore, the corollary would follow from showing that $\ker\pi''$ is { trivial}. 
Lemma~\ref{lem:prelim}(a) tells us that 
\[
\frac{\Lambda_n\otimes \cO_w}{\image h^u_{w,n}}\hookrightarrow \frac{\Lambda_{n-1}\otimes \cO_w}{\left(\image h^u_{w,n}\right)_{\Gamma_{n-1}}}.
\]
Since $\colwnu(H^1_{/f}(F_{n,w},T))= \image h^u_{w,n}$ by Proposition~\ref{prop:image},  the result follows.
\end{proof}

\begin{corollary}\label{cor:nablass}
Let  $z\in\HIw(F_{\infty,w},T)$ and  $u\in\Zp^\times$  a constant satisfying Lemma~\ref{lem:prelim}(b). For $n\ge1$, write $M_n$ for the $\Lambda_{\cO_w}$-module generated by the natural image of $z$ in  $H^1_{/f}(F_{n,w},T)$. Suppose that the natural projection \[
    \pi:\frac{H^1_{/f}(F_{n,w},T)}{M_n}\rightarrow \frac{H^1_{/f}(F_{n-1,w},T)}{M_{n-1}}
\]
has finite kernel for some $n$. Then,
\[
\nabla\left( \frac{H^1_{/f}(F_{n,w},T)}{M_n}\right)=[F_w:K_v]\cdot\ord_{\epsilon_n} \colwnu(z)(\epsilon_n).
\]
\end{corollary}
\begin{proof}
Since the projection $H^1_{/f}(F_{n,w},T)\rightarrow H^1_{/f}(F_{n-1,w},T)$ is surjective by Lemma~\ref{lem:surjcor}, we have
\[
\nabla\left(\frac{H^1_{/f}(F_n,T)}{M_n }\right)=\dim_{\Qp}\frac{H^1_{/f}(F_{n-1},T)}{M_{n-1} }+\length_{\Zp}\ker\pi.
\]
Let $\colwnu(z)=g_wh_w$, where $h_w=\gcd(\omega_{n-1},f_w)$. We deduce from Lemma~\ref{lem:prelim}(b) and Proposition~\ref{prop:nablawithcoeff} that
\[
\dim_{\Qp}\frac{H^1_{/f}(F_{n-1},T)}{M_{n-1} }=[F_w:K_v]\cdot \ord_{\epsilon_n}h_w(\epsilon_n).
\]
Applying Corollary~\ref{cor:comparepi}, we have
\[
\length_{\Zp}\ker\pi=[F_w:K_v]\cdot \ord_{\epsilon_n}g_w(\epsilon_n).
\]
Thus, putting everything together, we conclude that
\[
\nabla\left(\frac{H^1_{/f}(F_n,T)}{M_n }\right)=[F_w:K_v]\cdot(\ord_{\epsilon_n}h_w(\epsilon_n)+ \ord_{\epsilon_n}g_w(\epsilon_n))=[F_w:K_v]\cdot\ord_{\epsilon_n}\colwnu(z)(\epsilon_n)
\]
as required.
\end{proof}

We now explain how to calculate $\ord_{\epsilon_n} \colwnu(z)(\epsilon_n)$. 
In what follows, $\ord_p$ denotes the $p$-adic valuation on $\overline{\Qp}$ with $\ord_p(p) =1$.
{ Following \cite[Definition 4.4]{sprung13}, for a $2\times 2$ matrix $A=\begin{pmatrix} a & b \\ c & d \end{pmatrix}$ defined over $\overline{\Qp}$, we write
\[ \ord_p (A) = \begin{pmatrix} \ord_p(a) & \ord_p(b) \\ \ord_p(c) & \ord_p(d) \end{pmatrix}.\] }

\begin{proposition}\label{prop:evaluateH}
 Let $v\in \Sssp$ and write $r_v=\ord_p(a_v)\in\{1,\infty\}$ (thanks to the Weil's bound). For all $n\ge1$,  $$\ord_p\left(H_{v,n}(\epsilon_n)\right)=
\begin{cases}
\begin{pmatrix}
r_v+\sum_{i=1}^{\frac{n-1}{2}}\frac{1}{p^{2i}}&\sum_{i=1}^{\frac{n-1}{2}}\frac{1}{p^{2i-1}}\\
\infty&\infty
\end{pmatrix}
&\text{if $n$ is odd.}\\
\begin{pmatrix}
\sum_{i=1}^{\frac{n}{2}}\frac{1}{p^{2i-1}}&r_v+\sum_{i=1}^{\frac{n}{2}-1}\frac{1}{p^{2i}}\\
\infty&\infty
\end{pmatrix}&
\text{if $n$ is even.}
\end{cases}
$$
Let $z\in \HIw(F_{\infty,w},T)$ such that $\colwnu(z)(\epsilon_n)\ne 0$, then
\[
\ord_{\epsilon_n}\colwnu(z)(\epsilon_n)=
\ord_{\epsilon_n}\left(H_{v,n}^{\delta(w,n)}\col_{\delta(w,n),w}(z)(\epsilon_n)\right)
\]
where $\delta(w,n)\in\{\sharp,\flat\}$ depends on  the parity of $n$ once $w$ is fixed.
\end{proposition}
\begin{proof}
This is a special case of \cite[Proposition~4.6 and Corollary~4.8]{LLZ17}\footnote{There is a small typo in the statement of \cite[Proposition~4.6]{LLZ17} when $n$ is odd.} (with $k=2$ and $v$ in loc. cit. is taken to be $r_v$ here).
\end{proof}
\begin{remark}\label{rk:parity}
If $a_v=0$, then  for all $z$,
\[
\delta(w,n)=\begin{cases}
\flat&\text{if $n$ is odd,}\\
\sharp&\text{if $n$ is even.}
\end{cases}
\]
\end{remark}

\subsection{The ordinary case}\label{sec:ord}
We suppose in this section  that $v\in \Sord'$. We have a short exact sequence
\[ 0\lra \widehat{E}_v[p^{\infty}]\lra E[p^{\infty}] \lra \widetilde{E}_v[p^{\infty}]\lra 0, \]
where $\widehat{E}_v$ (resp., $\widetilde{E}_v$) is the formal group (resp., reduced curve) of $E$ at $K_v$. 
{This gives the exact sequence
\[ 0\lra \widehat{E}_v(K_{\infty,v})[p^{\infty}]\lra E(K_{\infty,v})[p^{\infty}] \lra \widetilde{E}_v(k_{\infty,v})[p^{\infty}], \]
where $k_{\infty,v}$ is the residue field of $K_{\infty, v}$. Note that 
$k_{\infty,v}$ is a finite field.}

In what follows, we write $\tT$ and $\wT$ for the $p$-adic Tate module of $\widetilde{E}$ and $\widehat{E}$ respectively.
\begin{lemma}\label{lem:localH2finite}
The modules $H^2_{\Iw}(K_{\infty,v}, \wT)$, $H^2_{\Iw}(K_{\infty,v},  T)$ and $H^2_{\Iw}(K_{\infty,v},  \tT)$ are finite.
\end{lemma}

\begin{proof}
{We first establish the finiteness of $\widehat{E}_v(K_{\infty,v})[p^{\infty}]$, $E(K_{\infty,v})[p^{\infty}]$ and $\widetilde{E}_v(k_{\infty,v})[p^{\infty}]$. Indeed, since $k_{\infty,v}$ is finite, so is $\widetilde{E}_v(k_{\infty,v})[p^{\infty}]$. The main theorem of \cite{Imai} says that $E(K_{\infty,v})[p^{\infty}]$ is finite, which in turn implies that $\widehat{E}(K_{\infty,v})[p^{\infty}]$ is also finite. Now, local duality tells us that $H^2_{\Iw}(K_{\infty,v}, \widetilde{T}) \cong H^0(K_{\infty,v}, \widehat{E}[p^{\infty}])^{\vee}$ and hence
the finiteness of $H^2_{\Iw}(K_{\infty,v}, \widetilde{T})$. The remaining finiteness assertions can be proven similarly.}
\end{proof} 

By \cite[Lemma 17.12]{kato04}, there is an injective map
\[  \col_v:H^1_{\Iw}(K_{\infty,v},\tT) {\lra} \Lambda\]
with finite cokernel. 
On taking Iwasawa cohomology of the short exact sequence
\[ 0\lra \wT\lra T \lra \tT\lra 0,\]
we obtain
\[ 0\lra H^1_{\Iw}(K_{\infty,v},\wT) \lra  H^1_{\Iw}(K_{\infty,v},T) \lra  H^1_{\Iw}(K_{\infty,v},\tT) \lra  H^2_{\Iw}(K_{\infty,v},\wT), \]
where one notes that $H^2_{\Iw}(K_{\infty,v},\wT)$ is finite by Lemma \ref{lem:localH2finite}.
Consequently, the following composition of maps
\[ H^1_{\Iw}(K_{\infty,v},T)  \lra  H^1_{\Iw}(K_{\infty,v},\tT) \stackrel{\mathrm{Col}_v}{\lra} \Lambda\]
 factors through to give an injection
\[  H^1_{/f, \Iw}(K_{\infty,v},T):=\frac{H^1_{\Iw}(K_{\infty,v},T)}{H^1_{\Iw}(K_{\infty,v}, \wT)} \hookrightarrow \Lambda\]
with finite cokernel. By an abuse notation, we shall write $\col_v$ for the composition above.
Note that for  $n\gg0$, we  have a short exact sequence
\begin{equation}\label{eq:B}
     0\lra B \lra H^1_{/f,\Iw}(K_{\infty,v},T)_{\Gamma_n} \stackrel{\col_v}{\lra} \Lambda /(\omega_n)\lra B\lra 0
\end{equation}
for some finite group $B$ that is independent of $n$.

Recall that
\[
H^1_{/f}(K_{n,v},T)=\frac{H^1(K_{n,v},T)}{E(K_{n,v})\otimes\Zp}=\frac{H^1(K_{n,v},T)}{H^1(K_{n,v}, \wT)}.\] We have the following diagram

\[   \xymatrixrowsep{0.25in}
\xymatrixcolsep{0.15in}\entrymodifiers={!! <0pt, .8ex>+} \SelectTips{eu}{}\xymatrix{
    & H^1_{\Iw}(K_{\infty,v}, \wT)_{\Gamma_{n}} \ar[d] \ar[r] &    H^1_{\Iw}(K_{\infty,v}, T)_{\Gamma_{n}}
    \ar[d] \ar[r]^(){} & H^1_{/f,\Iw}(K_{\infty,v}, T)_{\Gamma_{n}} \ar[d] \ar[r] &0\\
    0 \ar[r]^{} & H^1(K_{n,v}, \wT) \ar[r]^{} & H^1(K_{n,v}, T) \ar[r] & H^1_{/f}(K_{n,v}, T)\ar[r] &0 } \]
with exact rows. The first two vertical maps are injective with  cokernels $H^2_{\Iw}(K_{\infty,v}, \tT)$ and $H^2_{\Iw}(K_{\infty,v}, \tT)$ respectively. Taking Lemma \ref{lem:localH2finite} into consideration, it follows that the rightmost vertical map has finite kernel and cokernel which are bounded independently of $n$.

\begin{proposition} \label{kobayashi rank ordinary}
Let $z\in H^1_{\Iw}(K_{\infty,v}, T)$ such that $\col_v(z)\ne 0$ and write $M_n$ for the $\Lambda$-module generated by its image in $H^1_{/f}(K_{v,n}, T)$. When $n\gg0$,  $\nabla\left(\frac{H^1_{/f}(K_{n,v}, T)}{M_n} \right)$ is defined with
\[ \nabla\left(\frac{H^1_{/f}(K_{n,v}, T)}{M_n} \right) = \nabla\left(\frac{\Lambda}{(\omega_n, \col_v(z))} \right) =\ord_{\epsilon_n}\col_v(z)(\epsilon_n).\]
\end{proposition}

\begin{proof}
Let $M$ denote the image of $\Lambda\cdot z$ in the quotient $H^1_{/f,\Iw}(K_{\infty,v}, T)$. By the discussion just before the proposition, we have a map
\[ \left(\frac{H^1_{/f,\Iw}(K_{\infty,v}, T)}{M}\right)_{\Gamma_n}\lra \frac{H^1_{/f}(K_{n,v}, T)}{M_n}\]
with finite kernel and cokernel whose orders are independent of $n$ for $n\gg 0$. This in turn yields
\[ \nabla\Bigg( \left(\frac{H^1_{/f,\Iw}(K_{\infty,v}, T)}{M}\right)_{\Gamma_n} \Bigg) = \nabla\left(\frac{H^1_{/f}(K_{n,v}, T)}{M_n} \right).\]
Thanks to the short exact sequence \eqref{eq:B}, we have
\[ \nabla\Bigg( \left(\frac{H^1_{/f,\Iw}(K_{\infty,v}, T)}{M}\right)_{\Gamma_n} \Bigg) = \nabla\left(\frac{\Lambda}{(\omega_n, \col_v(z))} \right)\]
for $n\gg0$, resulting in the first equality of the proposition. The second equality follows from Proposition~\ref{prop:nabla}.
\end{proof}

\begin{corollary}\label{cor:nablaord}
Retain the setting of Proposition \ref{kobayashi rank ordinary}. Let $w\in \Sigma_\ord$ be a prime lying above $v$. Let $z'$ be the natural image of $z$ in  $\HIw(F_{\infty,w},T)=\HIw(K_{\infty,v},T)\otimes_{\Zp}\cO_w$. Let $M_n'$ be the $\Lambda_{\cO_w}$-module generated by $z'$ in the quotient $H^1_{/f}(F_{n,w}, T)$. Then for $n\gg0$, $\nabla\left(\frac{H^1_{/f}(F_{n,w}, T)}{M_n'} \right) $ is defined with
\[ \nabla\left(\frac{H^1_{/f}(F_{n,w}, T)}{M_n'} \right) = [F_w: K_v]\cdot \ord_{\epsilon_n}\col_w(z')(\epsilon_n),\]
where $\col_w$ is the Coleman map 
\[
\HIw(F_{\infty,w},T)\rightarrow \Lambda_{\cO_w}
\]
obtained from extending $\col_w$ $\cO_w$-linearly.
\end{corollary}

\begin{proof}
This follows from combining Propositions \ref{prop:nablawithcoeff} and \ref{kobayashi rank ordinary}.
\end{proof}

\section{Proofs of the main results}\label{sec:proofs}

Throughout this section, we assume that $(S1)-(S3)$. Let us write $d=[F:\QQ]$.
For each choice of $\fs=(s_w)_{w\in\Sss}\in\{\sharp,\flat\}^\Sss$, define 
\begin{equation}\label{eq:defnCol}
    \col_\fs:\bigoplus_{w\in\Sigma_p}\HIw(F_{\infty,w},T)\rightarrow \bigoplus_{w\in\Sigma_p}\Lambda_{\cO_w}\cong \Lambda^{\oplus d}
\end{equation}
to be the map given by $\col_{s_w,w}$ for $w\in \Sss$ and $\col_{w}$ for $w\in\Sord$, where $\col_w$ {is defined as in Corollary \ref{cor:nablaord}}.
\begin{lemma}\label{lem:injloc}
The localization map $$\loc_p:\HIw(F_\infty,T)\rightarrow \bigoplus_{w\in \Sigma_p}\HIw(F_{\infty,w},T)$$
is injective.
\end{lemma}
\begin{proof}
Recall that (S3) says that $\Sel^{\vec{\sharp}}(E/F_\infty)^\vee$ is $\Lambda$-torsion. By the Poitou-Tate exact sequence as given in \cite[Proposition~A.3.2]{perrinriou95}, we have the following exact sequence
\begin{equation}\label{eq:PTsharp}
    \HIw(F_\infty,T)\rightarrow \frac{\bigoplus_{w\in\Sigma_p}\HIw(F_{\infty,w},T)}{\ker\col_{\vec{\sharp}}}\rightarrow \Sel^{\vec{\sharp}}(E/F_\infty)^\vee.
\end{equation}
By Proposition~\ref{prop:image} and the pseudo-surjectivity of $\col_w$ for $w\in\Sord$ (see the discussion in \S\ref{sec:ord}), 
\[
\frac{\bigoplus_{w\in\Sigma_p}\HIw(F_{\infty,w},T)}{\ker\col_{\vec{\sharp}}}\cong \image\col_{\vec{\sharp}}
\]
is of rank $d$ over $\Lambda$. Given that $\Lambda$-module $\HIw(F_\infty,T)$ is free of rank $d$, the torsionness of   $\Sel^{\vec{\sharp}}(E/F_\infty)^\vee$ implies that the kernel of the first map in \eqref{eq:PTsharp} is $\Lambda$-torsion. But the $\Lambda$-module of $\HIw(F_\infty,T)$ is torsion-free  {by Lemma \ref{H1Iw}. Consequently}, the aforementioned kernel is trivial. This implies that $\ker\loc_p=0$.
\end{proof}

\begin{proposition}\label{prop:nicebasis}
 Let  $R_w$ denote the pre-image of $\HIw(F_{w,\infty},T)$ in $\HIw(F_\infty,T)$ for $w\in\Sigma_p$. There exist a family of elements  $c_w \in R_w$, $w\in \Sigma_p$ so that the quotient 
\[\frac{\HIw(F_\infty,T)}{\bigoplus_{w\in \Sigma_p} \Lambda_{\cO_w} c_w}\]
is a torsion $\Lambda$-module.
\end{proposition}
\begin{proof}
The injectivity of Lemma~\ref{lem:injloc} tells us that we have a direct sum of $\Lambda$-modules $\bigoplus_{w\in\Sigma_p}R_w$ inside $\HIw(F_{w,\infty},T)$. Furthermore, $R_w$ is a $\Lambda_{\cO_w}$-module of rank at most $1$ (since $\HIw(F_{\infty,w},T)$ is of rank $d$ over $\Lambda$). The proof of Lemma~\ref{lem:injloc} tells us that for each $w\in\Sigma_p$, the cokernel of the composition
\[
\HIw(F_\infty,T)\rightarrow \HIw(F_{\infty,w},T)\rightarrow \Lambda_{\cO_w}
\]
is $\Lambda$-torsion (where the second map is given by either $\col_{\flat,w}$ or $\col_w$ depending on whether $w\in\Sss$ or $w\in\Sord$). Thus, the $\Lambda_{\cO_w}$-rank of $R_w$ cannot be zero. Thus, we may pick  $c_w\in R_w$ so that $R_w/\Lambda_{\cO_w}c_w$ is $\Lambda_{\cO_w}$-torsion. The direct sum $\bigoplus_{w\in \Sigma_p} \Lambda_{\cO_w} c_w$ is then a $\Lambda$-module of rank $d$ as required.
\end{proof}

We fix a choice of $c_w$, $w\in\Sigma_p$ as given by Proposition~\ref{prop:nicebasis}. We write $M_\bc$ for the $\Lambda$-module $\bigoplus_{w\in \Sigma_p} \Lambda_{\cO_w} c_w$.
Recall that $\cY'(E/F_n)$ is defined to be $$\coker\left(\HIw(F_\infty,T)_{\Gamma_n}\rightarrow\bigoplus_{w\in\Sigma_p}H^1_{/f}(F_{n,w},T) \right).$$
We  define 
$$\cY''(E/F_n)=\coker\left((M_\bc)_{\Gamma_n}\rightarrow\bigoplus_{w\in\Sigma_p}H^1_{/f}(F_{n,w},T) \right).$$

\begin{proposition}\label{prop:Y''}
 For $n\gg0$, $\nabla_n\cY''(E/F_n)$ is defined. Furthermore, there exist $\vec{\sigma}=(\sigma_w),\vec{\tau}=(\tau_w)\in \{\sharp,\flat\}^\Sss$ such that when $n$ is odd and $n\gg0$, $\nabla_n\cY''(E/F_n)$ equals
\[
S(\vec{\sigma},n)+\nabla\left(\frac{\Lambda^{\oplus d}}{\col_{\vec{\sigma}}(M_{\bc})}\right)_{\Gamma_n},\]
where 
$$S(\vec{\sigma},n)=\phi(p^n)\left(\sum_{w:\sigma_w=\sharp}[F_w:K_v]\left(r_v+\sum_{i=1}^{\frac{n-1}{2}}\frac{1}{p^{2i}}\right)+\sum_{w:\sigma_w=\flat}[F_w:K_v]\sum_{i=1}^{\frac{n-1}{2}}\frac{1}{p^{2i-1}}\right),$$
whereas when $n$ is even and $n\gg0$, $\nabla_n\cY''(E/F_n)$ is given by
\[
T(\vec{\tau},n)+\nabla\left(\frac{\Lambda^{\oplus d}}{\col_{\vec{\tau}}(M_{\bc})}\right)_{\Gamma_n},
\]
where 
$$
T(\vec{\tau},n)=\phi(p^n)\left(\sum_{w:\tau_w=\sharp}[F_w:K_v]\sum_{i=1}^{\frac{n}{2}}\frac{1}{p^{2i-1}}+\sum_{w:\tau_w=\flat}[F_w:K_v]\left(r_v+\sum_{i=1}^{\frac{n}{2}-1}\frac{1}{p^{2i}}\right)\right),
$$
with $\phi$ being the Euler totient function. Here, $\col_{\vec{\sigma}}$ and $\col_{\vec{\tau}}$ are as defined by \eqref{eq:defnCol}.
\end{proposition}
\begin{proof}
We study the case where $n$ is odd. The proof for the case where $n$ is even is the same.

By definition, we have
\begin{equation}
    \cY''(E/F_n)=\bigoplus_{w\in\Sigma_p}\frac{H^1_{/f}(F_{n,w},T)}{(M_{w})_{\Gamma_n}},
\label{eq:directsum}
\end{equation}
where $M_w$ denotes the image of $\Lambda_{\cO_w} c_w$ in $\HIw(F_{\infty,w},T)$. Let $\pi_w$ be the natural map 
\[
\frac{H^1_{/f}(F_{n,w},T)}{(M_{w})_{\Gamma_n}}\rightarrow \frac{H^1_{/f}(F_{n-1,w},T)}{(M_{w})_{\Gamma_{n-1}}}.
\]
By Corollaries~\ref{cor:nablass} and \ref{cor:nablaord}, $\ker\pi_w$ is finite if and only if $\col_{\delta(w,n),w}(z_w)$ (resp. $\col_w(z_w)$) does not vanish at $\epsilon_n$ for $w\in\Sss$ (resp. $w\in\Sord$).

 Let $\vec{\sigma}=\left(\delta(w,n)\right)_{w\in\Sss}$, where $\delta_{w,n}$ is as given by  Proposition~\ref{prop:evaluateH}, which only depends on the parity of $n$).  On replacing $\fs$ in \eqref{eq:PTsharp} by $\vec{\sigma}$, the torsionness of $\Sel^{\vec{\sigma}}(E/F_\infty)^\vee$ forces $\col_{w,\delta_{w,n}}(z_w)$ (resp. $\col_{w}(z_w)$) to be a non-zero element of $\Lambda_{\cO_w}$  for $w\in\Sss$ (resp. $w\in\Sord$). In particular, when $n\gg0$, it does not vanish at $\epsilon_n$. Thus, $\ker\pi_w$ is finite and $\nabla\frac{H^1_{/f}(F_{n,w},T)}{(M_{w})_{\Gamma_n}}$ is defined. Its value can be calculated using Proposition~\ref{prop:evaluateH}, Corollaries~\ref{cor:nablass} and \ref{cor:nablaord}. To calculate $\nabla\cY''(E/F_n)$, we apply Lemma~\ref{lem:kobSES} to the direct sum in \eqref{eq:directsum}, which results in the formula as claimed.
\end{proof}

\begin{remark}
Suppose that $a_v=0$ for all $v\in \Sssp$, then Remark~\ref{rk:parity} tells us that
$  \vec{\sigma}=\vec{\flat}$ and $\vec{\tau}=\vec{\sharp}$.
Furthermore, the constants $S(\vec{\sigma},n)$ and $T(\vec{\tau},n)$ are given by
\begin{align*}
 S(\vec{\sigma},n)=  S(\vec{\flat},n)&=\phi(p^n)\sum_{w\in\Sss}[F_w:\Qp]\sum_{i=1}^{\frac{n-1}{2}}\frac{1}{p^{2i-1}}=d(p^{n-1}-p^{n-2}+p^{n-3}-\cdots -p),\\
  T(\vec{\tau},n)=T(\vec{\sharp},n)&=\phi(p^n)\sum_{w\in\Sss}[F_w:\Qp]\sum_{i=1}^{\frac{n}{2}}\frac{1}{p^{2i-1}}=d(p^{n-1} -p^{n-2}+p^{n-3}-\cdots -1).
\end{align*}
\end{remark}

\begin{corollary}\label{cor:nablaY'}
For $n\gg0$,  $\nabla\cY'(E/F_n)$ is defined and is given by
\[
\begin{cases}
S(\vec{\sigma},n)+\nabla\left(\coker \col_{\vec{\sigma}}\circ\loc_p\right)_{\Gamma_n}&\text{if $n$ is odd,} \\
T(\vec{\tau},n)+\nabla\left( \coker\col_{\vec{\tau}}\circ\loc_p\right)_{\Gamma_n}&\text{if $n$ is even,}
\end{cases}
\]
\end{corollary}
\begin{proof}
Again, we only treat the odd case. Consider the following short exact sequences, which are consequences of the third isomorphism theorem:
\[
0 \rightarrow\frac{ \HIw(F_\infty,T)_{\Gamma_n}}{(M_{\bc})_{\Gamma_n}}\rightarrow\left(\frac{\Lambda^{\oplus d}}{\col_{\vec{\sigma}}(M_{\bc})}\right)_{\Gamma_n}\rightarrow\left(\coker\col_{\vec{\sigma}}\circ\loc_p\right)_{\Gamma_n}\rightarrow 0,
\]
\[
0 \rightarrow \frac{ \HIw(F_\infty,T)_{\Gamma_n}}{(M_{\bc})_{\Gamma_n}}\rightarrow\cY''(E/F_n)\rightarrow \cY'(E/F_n)\rightarrow 0.
\]

Since $\HIw(F_\infty,T)/M_{\bc}$ is $\Lambda$-torsion, it follows that $\nabla\frac{ \HIw(F_\infty,T)_{\Gamma_n}}{(M_{\bc})_{\Gamma_n}}$ is defined for $n\gg0$. Thus,  we deduce from Proposition~\ref{prop:Y''} and Lemma~\ref{lem:kobSES} that $\nabla \frac{ \HIw(F_\infty,T)_{\Gamma_n}}{(M_{\bc})_{\Gamma_n}}$ is defined and is given by
\[
\nabla\frac{ \HIw(F_\infty,T)_{\Gamma_n}}{(M_{\bc})_{\Gamma_n}}=\nabla \cY''(E/F_n)- \nabla\cY'(E/F_n).
\]

But Proposition~\ref{prop:Y''} also tells us that $\nabla \left(\frac{\Lambda^{\oplus d}}{\col_{\vec{\sigma}}(M_{\bc})}\right)_{\Gamma_n}$ is defined for $n\gg0$. Thus, $\nabla\coker\col_{\vec{\sigma}}\circ\loc_p$ is defined and satisfies
\[
\nabla \left(\coker\col_{\vec{\sigma}}\circ\loc_p\right)_{\Gamma_n}=\nabla \left(\frac{\Lambda^{\oplus d}}{\col_{\vec{\sigma}}(M_{\bc})}\right)_{\Gamma_n}-\nabla \frac{ \HIw(F_\infty,T)_{\Gamma_n}}{(M_{\bc})_{\Gamma_n}}.
\]
The result now follows from combining these two equations with the formula given in Proposition~\ref{prop:Y''}.
\end{proof}

\begin{defn}
For $0\le n\le \infty$, we define
\[
\Sel^0(E/F_n):=\ker\left(\Sel_{p^\infty}(E/F_n)\rightarrow \prod_{w\in \Sigma_p}H^1(F_{n,w},E[p^\infty])\right).
\]
Equivalently, we have
\[
\Sel^0(E/F_n):=\ker\left(H^1(G_{\Sigma}(F_n),E[p^{\infty}])\rightarrow \prod_{w\in \Sigma(F_n)}H^1(F_{n,w},E[p^\infty])\right),
\] where $\Sigma(F_n)$ denotes the set of primes of $F_n$ above $\Sigma$.
The Pontryagin duals of $\Sel_{p^\infty}(E/F_n)$ and $\Sel^0(E/F_n)$ are denoted by $\cX(E/F_n)$ and $\cX^0(E/F_n)$ respectively.
\end{defn}

\begin{lemma}\label{lem:fineSelmercontrol}
The natural restriction map
\[
\Sel^0(E/F_n) \lra \Sel^0(E/F_{\infty})^{\Gamma_n}
\] has finite kernel and cokernel which are bounded independently of $n$.
\end{lemma}

\begin{proof}  { In the proof of Lemma \ref{H1Iw}, we have seen that $E(F_{\infty})[p^\infty]=0$. It then follows} that the middle map of the following commutative diagram
\[   \xymatrixrowsep{0.25in}
\xymatrixcolsep{0.15in}\entrymodifiers={!! <0pt, .8ex>+} \SelectTips{eu}{}\xymatrix{
    0 \ar[r]^{} & \Sel^0(E/F_n) \ar[d] \ar[r] &  H^1(G_{\Sigma}(F_n),E[p^{\infty}])
    \ar[d]^{} \ar[r]^(){} & \displaystyle\prod_{w\in \Sigma(F_n)}H^1(F_{n,w},E[p^\infty]) \ar[d]^{}\\
    0 \ar[r]^{} & \Sel^0(E/F_{\infty})^{\Gamma_{n}}\ar[r]^{} &  H^1(G_{\Sigma}(F_\infty),E[p^{\infty}])^{\Gamma_n} \ar[r] &\displaystyle \left(\prod_{w\in \Sigma(F_{\infty})}H^1(F_{\infty,w},E[p^\infty])\right)^{\Gamma_n}}
    \] is an isomorphism  via a Hochschild-Serre spectral sequence argument. Hence it suffices to show that the rightmost map has finite kernel which is bounded independent of $n$. For primes not dividing $p$, this is discussed in \cite[Lemma 3.3]{G99}. It therefore remains to consider the primes above $p$. Let $w$ be such a prime. Then the kernel of the restriction map
    \[ H^1(F_{n,w},E[p^{\infty}])\lra H^1(F_{\infty,w},E[p^{\infty}])^{\Gamma_n}\]
    is  given by $H^1(\Gamma_n, E(F_{\infty,w})[p^{\infty}])$ by the Hochschild-Serre spectral sequence. Since our elliptic curve is assumed to have good reduction at all primes above $p$, the main theorem of \cite{Imai} says that $E(F_{\infty,w})[p^{\infty}]$ is finite.  The finiteness and boundedness of the kernel now follow. This completes the proof of the lemma.
\end{proof} 

\begin{lemma}\label{lem:X0}
For $n\ge0$:
\begin{itemize}
    \item[(a)]We have a short exact sequence
\[
0\rightarrow\cY(E/F_n)\rightarrow \cX(E/F_n)\rightarrow \cX^0(E/F_n)\rightarrow 0.
\]
\item[(b)] For $n\gg0$, $\nabla\cX^0(E/F_n)$ is defined satisfying the equality
\[
\nabla\cX^0(E/F_n)=\nabla\cX^0(E/F_\infty)_{\Gamma_n}.
\]
\end{itemize}
\end{lemma}
\begin{proof}
Part (a) is a consequence of the Poitou-Tate exact sequence. See  \cite[(10.35)]{kobayashi03}.

Again by the Poitou-Tate exact sequence, the cokernel of the last map in \eqref{eq:PTsharp} is isomorphic to $\cX^0(E/F_\infty)$. This implies that $\cX^0(E/F_\infty)$ is $\Lambda$-torsion.  By Lemma \ref{lem:fineSelmercontrol}, the kernel and cokernel of the natural map
\[
\cX^0(E/F_\infty)_{\Gamma_n}\rightarrow \cX^0(E/F_n)
\]
are finite and bounded independent of $n$. Part (b) now follows from {combining the latter observation with} Proposition~\ref{prop:nabla}(b) (see \cite[Proposition 10.6ii)]{kobayashi03} for the proof when $F=\QQ$).
\end{proof}

We are now ready to prove the main theorems of the paper.

\begin{theorem}
Under hypotheses (S1)-(S3), we have
\begin{itemize}
    \item[(A)] $\rank_{\ZZ}E(F_n)$ is bounded independently of $n$;
    \item[(B)] Suppose that $\sha_p(E/F_n)$ is finite for all $n\ge0$. Let $r_\infty=\lim_{n\rightarrow\infty}\rank_\ZZ E(F_n)$. Then, for $n\gg0$, we have
    $$\nabla_n\sha_p(E/F_n)=
    \begin{cases}
    S(\vec{\sigma},n)+\phi(p^n)\mu_{\vec{\sigma}}+\lambda_{\vec{\sigma}}-r_\infty&\text{if $n$ is odd,}\\
    T(\vec{\tau},n)+\phi(p^n)\mu_{\vec{\tau}}+\lambda_{\vec{\tau}}-r_\infty&\text{if $n$ is even.}
    \end{cases}
    $$
\end{itemize}
\end{theorem}
\begin{proof}
Corollary~\ref{cor:nablaY'} tells us that $\rank_{\Zp}\cY'(E/F_n)$ is bounded independently of $n$. Thus, the same is true for $\cY(E/F_n)$ thanks to Proposition~\ref{prop:equivY}. { Also,  $\rank_{\Zp}\cX^0(E/F_n)$ is bounded independently of $n$ by Lemma \ref{lem:X0}(b). Thus, by the short exact sequence in Lemma~\ref{lem:X0}(a), we have} that $\rank_{\Zp}\cX(E/F_n)$ is bounded independently of $n$. Hence, part (A) now follows from the well-known exact sequence
\begin{equation}\label{eq:tautology}
  0\rightarrow E(F_n)\otimes\Qp/\Zp\rightarrow \Sel_{p^\infty}(E/F_n)\rightarrow \sha_p(E/F_n)\rightarrow 0  .
\end{equation}

By Lemma~\ref{lem:kobSES}, the short exact sequence \eqref{eq:tautology} implies that
\[
\nabla\sha_p(E/F_n)=\nabla \cX(E/F_n)-\nabla E(F_n)\otimes\Zp.
\]
Part (A) tells us that $\nabla E(F_n)\otimes\Zp=r_\infty$ for $n\gg0$. It remains to calculate $\nabla\cX(E/F_n)$.

Let $\fs\in \{\sharp,\flat\}^\Sss$. On replacing $\vec{\sharp}$ in the Poitou-Tate exact sequence \eqref{eq:PTsharp} by $\fs$, we have the exact sequence
\[
0\rightarrow\coker \col_{\fs}\circ\loc_p\rightarrow \Sel^\fs(E/F_\infty)^\vee\rightarrow\cX^0(E/F_\infty)\rightarrow 0.
\]
Taking $\Gamma_n$-invariant, we obtain a six terms exact sequence
\[0 \rightarrow H_1\big(\Gamma_n, \coker \col_{\fs}\circ\loc_p\big)\rightarrow H_1\left(\Gamma_n, \Sel^\fs(E/F_\infty)^\vee\right)
\rightarrow H_1(\Gamma_n,\cX^0(E/F_\infty))\rightarrow \left(\coker \col_{\fs}\circ\loc_p\right)_{\Gamma_n} \] \[\rightarrow \Sel^\fs(E/F_\infty)^\vee_{\Gamma_n}\rightarrow\cX^0(E/F_\infty)_{\Gamma_n}\rightarrow 0.
\]
 Note that $H_1(\Gamma_n,-)= (-)^{\Gamma_n}$ and one can easily verify that the transition maps on these terms are given by multiplication by $1+\gamma_n+\cdots +\gamma_n^{p-1}$. We may therefore apply Lemma \ref{lemma:NSWfg} to conclude that the Kobayashi ranks of the leftmost three terms vanish for $n\gg 0$. Lemma \ref{lem:kobSES} thus implies the equality 
\[
\nabla \cX^0(E/F_\infty)_{\Gamma_n}+ \nabla \left(\coker \col_{\fs}\circ\loc_p\right)_{\Gamma_n}=\nabla\Sel^\fs(E/F_\infty)^\vee_{\Gamma_n}
\]
for $n\gg 0$.
 Thus, on applying Proposition~\ref{prop:nabla}(b) to the torsion $\Lambda$-module $\Sel^\fs(E/F_\infty)^\vee$, we obtain
\[
\nabla \cX^0(E/F_\infty)_{\Gamma_n}=\nabla\left(\coker \col_{\fs}\circ\loc_p\right)_{\Gamma_n}-\phi(p^n)\mu_{\fs}-\lambda_{\fs}.
\]
Combining this with Lemma~\ref{lem:X0}, Proposition~\ref{prop:equivY} and Corollary~\ref{cor:nablaY'}, we deduce that
\[
\nabla\cX(E/F_n)= \begin{cases}
    S(\vec{\sigma},n)+\phi(p^n)\mu_{\vec{\sigma}}+\lambda_{\vec{\sigma}}&\text{if $n$ is odd,}\\
    T(\vec{\tau},n)+\phi(p^n)\mu_{\vec{\tau}}+\lambda_{\vec{\tau}}&\text{if $n$ is even.}
    \end{cases}
\]
Hence the result.
\end{proof}

\bibliographystyle{amsalpha}
\bibliography{references}

\end{document}